\documentclass[12pt]{amsart}

\usepackage{amsmath}
\usepackage{amssymb}
\usepackage[all]{xy}
\usepackage{longtable}
\usepackage[osf,sc]{mathpazo}
\usepackage{euscript}
\usepackage{calrsfs}
\usepackage{array}

\newtheorem{theorem}[equation]{Theorem}
\newtheorem{lemma}[equation]{Lemma}
\newtheorem{proposition}[equation]{Proposition}

\theoremstyle{remark}
\newtheorem{remark}[equation]{Remark}

\numberwithin{equation}{subsection}

\newcommand{\FF}{\mathbb{F}}
\newcommand{\ZZ}{\mathbb{Z}}
\newcommand{\QQ}{\mathbb{Q}}

\newcommand{\NN}{\mathbb{N}}

\newcommand{\bu}{\mathbf{u}}
\newcommand{\bv}{\mathbf{v}}

\newcommand{\bj}{\mathbf{j}}

\newcommand{\cA}{\mathcal{A}}

\newcommand{\cZ}{\mathcal{Z}}

\DeclareMathAlphabet{\matheur}{U}{eur}{m}{n}

\newcommand{\fs}{\mathfrak{s}}

\newcommand{\fz}{\mathfrak{z}}

 \DeclareMathOperator{\wt}{wt}
\DeclareMathOperator{\Li}{Li}

\begin{document}
\title[On finite Carlitz multiple polylogarithms]{{\large{On finite Carlitz multiple polylogarithms}}}

\author{Chieh-Yu Chang}
\address{Department of Mathematics, National Tsing Hua University, Hsinchu City 30042, Taiwan
  R.O.C.}

\email{cychang@math.nthu.edu.tw}

\author{Yoshinori Mishiba }
\address{Department of General Education,
National Institute of Technology, Oyama College, Japan}
\email{mishiba@oyama-ct.ac.jp }

\thanks{The first author was partially supported by MOST Grant
  102-2115-M-007-013-MY5.  }
\thanks{The second author is supported by JSPS KAKENHI Grant Number 15K17525.}


\date{November 8, 2016}

\begin{abstract}
In this paper, we define finite Carlitz multiple polylogarithms and show that every finite multiple zeta value over the rational function field $\FF_{q}(\theta)$ is an  $\FF_{q}(\theta)$-linear combination of finite Carlitz multiple polylogarithms at integral points. It is completely compatible with the formula for Thakur MZV's established in \cite{C14}.
\end{abstract}


\maketitle

\section{Introduction}
Let $A:=\FF_{q}[\theta]$ be the polynomial ring in the variable $\theta$ over the finite field $\FF_{q}$ of $q$ elements with characteristic $p$, and $k$ be the quotient field of $A$. We denote by $k_{\infty}$ the completion of $k$ with respect to the place at infinite. We denote by $A_{+}$ the set of monic polynomials in $A$.

The characteristic $p$ multiple zeta values (abbreviated as MZV's) were introduced by Thakur~\cite{T04}: for $\fs=(s_{1},\ldots,s_{r})\in \NN^{r}$,
\begin{equation}\label{E:DefMZV}
\zeta_{A}(s_{1},\ldots,s_{r}):=\sum \frac{1}{a_{1}^{s_{1}}\cdots a_{r}^{s_{r}}  }\in k_{\infty},
\end{equation}
where $a_{1},\ldots,a_{r}$ run over all monic polynomials in $A$ satisfying \[\deg_{\theta} a_{1}>\deg_{\theta} a_{2}>\cdots> \deg_{\theta} a_{r}.\] The values above play the positive characteristic analogue of classical multiple zeta values (see~\cite{Zh16}), and they are in fact non-vanishing by Thakur~\cite{T09}. One knows further that MZV's occur as periods of certain mixed Carlitz-Tate $t$-motives (see~\cite{AT09}).

In the seminal paper~\cite{AT90}, Anderson and Thakur introduced the $n$th tensor power of the Carlitz module and established a deep connection between $\zeta_{A}(n)$ and the $n$th Carlitz polylogarithm for each positive integer $n$. The $n$th Carlitz polylogarithm is the function field analogue of the classical $n$th  polylogarithm defined by the series
\[ \Li_{n}(z):=\sum_{i=0}^{\infty} \frac{z^{q^{i}}}{L_{i}^{n}},  \]
where $L_{0}:=1$ and $L_{i}:=(\theta-\theta^{q})\cdots (\theta-\theta^{q^{i}})$ for $i\in \NN$. When $n=1$, the series above is the Carlitz logarithm (see~\cite{Ca35, Go96, T04}). What Anderson and Thakur showed is that $\zeta_{A}(n)$ is a $k$-linear combination of $\Li_{n}$ at some integral points in $A$.

Inspired by the classical multiple polylogarithms (see~\cite{W02, Zh16}), the first author of the present paper defined for each $\fs=(s_{1},\ldots,s_{r})\in \NN^{r}$ the $\fs$th Carlitz multiple polylogarithm (abbreviated as CMPL):
\begin{equation}\label{E:CMPL}
 \Li_{\fs}(z_{1},\ldots,z_{r}):=\sum_{i_{1}>\cdots>i_{r}\geq 0} \frac{z_{1}^{q^{i_{1}}}\cdots z_{r}^{q^{i_{r}}}  }{L_{i_{1}}^{s_{1}}\cdots L_{i_{r}}^{s_{r}}   } .
\end{equation}
Note that in the classical setting, there is a simple identity that a multiple zeta value $\zeta(\fs)$ is the specialization of  the $\fs$th multiple polylogarithm (several variables) at $(1,\ldots,1)$. Using the theory of Anderson-Thakur polynomials~\cite{AT90} the first author~\cite{C14} derived an explicit formula expressing $\zeta_{A}(\fs)$ as a $k$-linear combination of $\Li_{\fs}$ at some integral point (see Theorem~\ref{T:Chang}) generalizing Anderson-Thakur's work to arbitrary depth.

The study of this paper is inspired by the work of Kaneko and Zagier~\cite{KZ} on finite multiple zeta values, which are in the $\QQ$-algebra
\[\cA:=\prod_{p}\ZZ /(p) \big/ \bigoplus_{p}\ZZ /(p) ,\]
where $p$ runs over all prime numbers. In analogy with $\cA$, it is natural to define the $k$-algebra $\cA_{k}$ (see~(\ref{E:DefcAk})). One then naturally defines a finite version of Thakur MZV's (\ref{E:DefMZV}), which we (also) call finite multiple zeta values (abbreviated as FMZV's) denoted by $\zeta_{\cA_{k}}(\fs)$. See (\ref{E:DefFMZV}) for the definition and note that Thakur also defines FMZV's in \cite{T16} (see also a variant in \cite{PP15}). In this paper we define a finite version of CMPL's (\ref{E:CMPL}), called finite Carlitz multiple polylogarithms (abbreviated as FCMPL's) and denoted by $\Li_{\cA_{k},\fs}(z_{1},\ldots,z_{r})$ for $\fs\in \NN^{r}$ (see (\ref{E:DefFCMPL}) for the precise definition). We then have that the FCMPL's satisfy the stuffle relations (see \S~\ref{Sec:stuffle}). The main result in this paper is to establish an explicit formula expressing each FMZV $\zeta_{\cA_{k}}(\fs)$  as a $k$-linear combination of $\Li_{\cA_{k},\fs}$ at some integral points (see Theorem~\ref{T:MainThm}). It is interesting that the formula for $\zeta_{\cA_{k}}(\fs)$ completely matches with the formula for $\zeta_{A}(\fs)$ (cf.~Theorem~\ref{T:Chang} and Theorem~\ref{T:MainThm}), and its proof highly relies on the theory of Anderson-Thakur polynomials~\cite{AT90}.

At the end of the introduction, we give a list of some interesting problems for future research.

\begin{enumerate}
\item[$\bullet$] Connection bwteen Thakur MZV's and FMZV's (cf.~\cite{KZ}).
\item[$\bullet$] Non-vanishing problems for FMZV's (cf.~\cite{ANDTR16, T09}).
\item[$\bullet$] Logarithmic and period interpretation of FCMPL's and FMZV's (cf.~\cite{AT90, AT09}).
\item[$\bullet$] Transcendence theory for  FCMPL's and FMZV's (cf.~\cite{ABP04, P08, C14, C16, CM16, CP12, CY07, M14, Yu91, Yu97}).
\item[$\bullet$] Relation between FCMPL's and $t$-motives (cf.~\cite{AT90, AT09, C14}).
\end{enumerate}

\section{Finite multiple zeta values}
\subsection{The definition of FMZV's}

Following Kaneko and Zagier, we define the $k$-algebra
\begin{equation}\label{E:DefcAk}
\mathcal{A}_{k}:=\prod_{P}A/(P) \big/  \bigoplus_{P}A/(P)  ,
\end{equation}
where $P$ runs over all monic irreducible polynomials in $A$. In analogy with classical finite MZV's, one considers the following finite version of ($\infty$-adic) Thakur MZV's denoted by $\zeta_{\mathcal{A}_{k}}(s_{1},\ldots,s_{r})$ for any $r$-tuple $(s_{1},\ldots,s_{r})\in \NN^{r}$. One first defines for a monic irreducible polynomial $P\in A$,
\[\zeta_{\mathcal{A}_{k}}(s_{1},\cdots,s_{r})_{P}:=\sum \frac{1}{a_{1}^{s_{1}}\cdots a_{r}^{s_{r}}} \hbox{ mod }P \in A/(P),\]
where the sum runs over all monic polynomials $a_{1},\ldots,a_{r}\in A$ satisfying
\[\deg P> \deg a_{1}>\cdots>\deg a_{r}  .\]
One then defines the finite multiple zeta value abbreviated as FMZV (see also~\cite{T16}):
\begin{equation}\label{E:DefFMZV}
\zeta_{\mathcal{A}_{k}}(s_{1},\cdots,s_{r}):=(\zeta_{\mathcal{A}_{k}}(s_{1},\cdots,s_{r})_{P} )\in \mathcal{A}_{k}.
\end{equation}
We call $r$ the depth and $\wt(\fs):=\sum_{i=1}^{r}s_{i}$ the weight of the presentation $\zeta_{\cA_{k}}(\fs)$.

The motivation of our study in this paper comes from the identity in \cite{C14} that any ($\infty$-adic) Thakur  MZV is a $k$-linear combination of Carlitz multiple polylogarithms (abbreviated as CMPL's) at integral points (generalization of the formula of Anderson-Thakur~\cite{AT90} for the depth one case). Our main result is to establish the same identity for the FMZV's.

\subsection{The algebra of FMZV's} In \cite{T10}, Thakur proved that the $\FF_{p}$-vector space spanned by MZV's forms an algebra. Using Thakur's theory~\cite{T10}, one finds the same phenomenon for FMZV's in the following theorem. In other words, the $k$-vector space spanned by FMZV's forms a $k$-algebra that is defined over $\FF_{p}$.

\begin{proposition}\label{P:FMZVAlgebra}
Let $\cZ\subseteq \cA_{k}$ be the $\FF_{p}$-vector subspace spanned by all FMZV's. Then $\cZ$ forms an $\FF_{p}$-algebra.
\end{proposition}
\begin{proof} It suffices to show that for arbitrary $\fs\in \NN^{r}$ and $\fs'\in \NN^{r'}$, there exists $\fs_{1},\ldots,\fs_{m}\in \cup_{\ell}\NN^{\ell}$ with ${\rm{wt}}(\fs_{i})={\rm{wt}}(\fs)+{\rm{wt}}(\fs')$, and $f_{1},\ldots,f_{m}\in \FF_{p}$ so that
\[  \zeta_{\mathcal{A}_{k}}(\fs)_{P} \zeta_{\mathcal{A}_{k}}(\fs')_{P}=\sum_{i=1}^{m}f_{i} \zeta_{\mathcal{A}_{k}}(\fs_{i})_{P}\in A/(P)\]
for all primes $P\in A_{+}$.

For any $r$-tuple $\fs=(s_{1},\ldots,s_{r})$ and $d \in \NN$, we put
\[S_{<d}(\fs):=\sum \frac{1}{a_{1}^{s_{1}}\cdots a_{r}^{s_{r}}} \in k,\]
where the sum runs over all monic polynomials $a_{1},\ldots,a_{r}\in A$ satisfying
\[d > \deg a_{1}>\cdots>\deg a_{r}  .\]
It follows that
\begin{equation}\label{E:zetaSd}
\zeta_{\cA_{k}}(\fs)_{P} =S_{< \deg P}(\fs) \hbox{ mod }P .
\end{equation}
Note that \cite[Cor.~2.2.10]{To15} implies that $S_{<\deg P}(\fs) S_{< \deg P}(\fs')$ is an $\FF_{p}$-linear combination of some $S_{<\deg P}(\fs'')$ with $\wt(\fs'')=\wt(\fs)+\wt(\fs')$,
where the $\fs''$'s and the coefficients in $\FF_{p}$ are independent of $\deg P$, whence the desired result by modulo $P$.
\end{proof}

\begin{remark}
The authors were informed by Thakur that his student Shuhui Shi has derived several identities on these FMZV's with $k$-coefficients, including
 Proposition~\ref{P:FMZVAlgebra}.
\end{remark}

\begin{remark}
The authors were informed by H.-J. Chen that as such the case above, the techniques in \cite{Chen15} can be used to derive an explicit formula for the product of two finite single zeta values in terms of linear combinations of some FMZV's.
\end{remark}

\section{Finite Carlitz multiple polylogarithms and the main result}

In what follows, for any tuple $\fs\in \NN^{r}$ we define its associated finite Carlitz multiple polylogartihm (abbreviated as FCML)
\[ \Li_{\mathcal{A}_{k},\fs}:k^{r}\rightarrow \mathcal{A}_{k} .\]
 Fixing any $r$-tuple $\fs=(s_{1},\ldots,s_{r})\in \NN^{r}$ and an $r$-tuple of independent variables $\fz=(z_{1},\ldots,z_{r})$, we define
 the quotient ring
 \[\cA_{k,\fz}:=\prod_{P} A[\fz]/(P) \big/ \bigoplus_{P} A[\fz]/(P),\]
 where \[ A[\fz]=A[z_{1},\ldots,z_{r}]  .\]
 We then define
\begin{equation}\label{E:DefFCMPL}
\Li_{\mathcal{A}_{k},\fs}(\fz):=\left( \Li_{\mathcal{A}_{k},\fs}(z_{1},\ldots,z_{r})_{P} \right)\in \cA_{k,\fz} ,
\end{equation}
where
\[\Li_{\mathcal{A}_{k},\fs}(z_{1},\ldots,z_{r})_{P}:=\sum_{\deg P>i_{1}>\cdots> i_{r}\geq 0 }\frac{z_{1}^{q^{i_{1}}}\cdots z_{r}^{q^{i_{r}}}}{L_{i_{1}}^{s_{1}}\cdots L_{i_{r}}^{s_{r}}} \hbox{ mod }P \in A[z_{1},\ldots,z_{r}]/(P)  .\]
We note that $P$ does not divide $(\theta^{q^{i}}-\theta)$ if and only if $\deg P\nmid i$, and hence $\Li_{\cA_{k},\fs}(\fz)$ is well-defined in $\cA_{k,\fz}$. Furthermore, $\Li_{\cA_{k},\fs}(\bu)$ is well-defined in $\cA_{k}$ for any $\bu=(u_{1},\ldots,u_{r})\in k^{r}$ since $\Li_{\cA_{k},\fs}(\bu)_{P}$ is defined in $A/(P)$ for $P$ not dividing the denominators of $u_{1},\ldots,u_{r}$. Such as the $\infty$-adic case, we call $r$ the depth and $\wt(\fs)$ the weight of the presentation $\Li_{\cA_{k},\fs}(\bu)$.

\subsection{Stuffle relations}\label{Sec:stuffle}
Let $\fz'=(z_{1}',\ldots,z_{r'}')$ be an $r'$-tuple of variables independent from the $z_{i}$'s of $\fz$. For each prime $P\in A_{+}$ we consider the natural multiplication map
\[A[\fz]/(P)\times A[\fz']/(P) \rightarrow A[\fz,\fz']/(P)   ,\]
which induces the following map
\begin{equation}\label{E:MulAkz}
\cA_{k,\fz}\times \cA_{k,\fz'}\rightarrow \cA_{k,(\fz,\fz')}.
\end{equation}
We denote by $\Li_{\cA_{k},\fs}(\fz)\cdot \Li_{\cA_{k},\fs'}(\fz')\in \cA_{k,(\fz,\fz')}$ the image of $(\Li_{\cA_{k},\fs}(\fz), \Li_{\cA_{k},\fs'}(\fz'))\in \cA_{k,\fz}\times \cA_{k,\fz'}$ under the map~(\ref{E:MulAkz}).

Note that since the indexes of the finite sum $\Li_{\cA_{k},\fs}(\fz)_{P}$ are in the total ordered set $\ZZ_{\geq 0}$, the classical stuffle relations (for multiple polylogarithms) work here by componentwise multiplication. We describe the details as the following.

Given $\mathfrak{s}=(s_{1},\ldots,s_{r})\in \NN^{r}$ and $\mathfrak{s}'=(s_{1}',\ldots,s_{r'}')\in \NN^{r'}$, we fix a positive integer $r''$ with ${\rm{max}}\left\{ r,r' \right\}\leq r''\leq r+r'$. We consider a pair consisting of two vectors $\bv, \bv'\in  \ZZ_{\geq 0}^{r''}$ which are required to satisfy $\bv+\bv'\in \NN^{r''}$ and which are obtained from the following ways. One vector $\bv$ is obtained  from $\mathfrak{s}$ by inserting $(r''-r)$ zeros in all possible ways (including in front and at end), and another vector $\bv'$ is obtained from $\mathfrak{s}'$ by inserting $(r''-r')$ zeros in all possible ways (including in front and at end).

 One observes from the definition  that FCMPL's satisfy the stuffle relations which are analogous to the classical case (cf. \cite{W02}):
\begin{equation}\label{E:stuffle}
  \Li_{\cA_{k},\fs}(\fz)\cdot \Li_{\cA_{k},\fs'}(\fz') =\sum_{(\bv,\bv')}\Li_{\cA_{k}, \bv+\bv'}(\fz''),
\end{equation}where the pair $(\bv,\bv')$ runs over all the possible expressions as above for all $r''$ with ${\rm{max}}\left\{ r,r' \right\}\leq r''\leq r+r'$. For each such $\bv+\bv'\in \NN^{r''}$, the component $z_{i}''$ of $\fz''$ is $z_{j}$ if the $i$th component of $\bv$ is $s_{j}$ and the $i$th component of $\bv'$ is $0$, it is $z_{\ell}'$ if the $i$th component of $\bv$ is $0$ and the $i$th component of $\bv'$ is $s_{\ell}'$, and finally it is $z_{j}z_{\ell}'$ if the $i$th component of $\bv$ is $s_{j}$ and the $i$th component of $\bv'$ is $s_{\ell}'$.

For example, for $r=r'=1$ (\ref{E:stuffle}) yields
\[ {\rm{Li}}_{\cA_{k},s}(z) \cdot {\rm{Li}}_{\cA_{k},s'}(z')={\rm{Li}}_{\cA_{k},(s,s')}(z,z')+ {\rm{Li}}_{\cA_{k},(s',s)}(z',z)+{\rm{Li}}_{\cA_{k}, s+s'}(zz'). \]
 For $r=1,r'=2$, one has
\[
    \begin{array}{rl}
     {\rm{Li}}_{\cA_{k}, s}(z) \cdot {\rm{Li}}_{\cA_{k},(s_{1}',s_{2}')}(z_{1}',z_{2}')  &={\rm{Li}}_{\cA_{k},(s,s_{1}',s_{2}')}(z,z_{1}',z_{2}')+ {\rm{Li}}_{\cA_{k},(s_{1}',s,s_{2}')}(z_{1}',z,z_{2}')+{\rm{Li}}_{\cA_{k},(s_{1}',s_{2}',s)}(z_{1}',z_{2}',z) \\
       & +{\rm{Li}}_{\cA_{k},(s+s_{1}',s_{2}')}(z z_{1}',z_{2}')+ {\rm{Li}}_{\cA_{k},(s_{1}',s+s_{2}')}(z_{1}',zz_{2}').\\
    \end{array}
  \]
\begin{remark}
From the stuffle relations above, we see that the product of $\Li_{\cA_{k},\fs}(\bu)$ and $\Li_{\cA_{k},\fs'}(\bu')$ is an $\FF_{p}$-linear combinations of some FCMPL's of the same weight $\wt(\fs)+\wt(\fs')$ at rational points over $k$.
\end{remark}
\subsection{The formula for Thakur MZV's} Let $t, x, y$ be new independent variables. We put $G_{0}(y):=1$ and define polynomials $G_{n}(y)\in \FF_{q}[t,y]$ for $n\in \NN$ by the product
\[
G_{n}(y)=\prod_{i=1}^{n}\left( t^{q^{n}}-y^{q^{i}} \right).
\]
For a non-negative integer $n$, we express $n = \sum n_{i} q^{i}$ ($0 \leq n_{i} \leq q-1$) as the base $q$-expansion. We define the Carlitz factorial $\Gamma_{n+1} := \prod D_{i}^{n_{i}}$, where $D_{0} := 1$ and $D_{i} := \prod_{j=0}^{i-1} (\theta^{q^{i}}-\theta^{q^{j}})$ for $i \in \NN$.
For $n=0,1,2,\ldots$, we define the sequence of Anderson-Thakur polynomials $H_{n}\in A[t]$ by the generating function identity
\[
\left( 1-\sum_{i=0}^{\infty} \frac{ G_{i}(\theta) }{ D_{i}|_{\theta=t}} x^{q^{i}}  \right)^{-1}=\sum_{n=0}^{\infty} \frac{H_{n}}{\Gamma_{n+1}|_{\theta=t}} x^{n}.
\]

In what follows, we fix an $r$-tuple of positive integers $\fs=(s_{1},\ldots,s_{r})\in \NN^{r}$. For each $1\leq i\leq r$, we expand the Anderson-Thakur polynomial $H_{s_{i}-1}\in A[t]$ as
\begin{equation}\label{E:expansionHs}
H_{s_{i}-1}=\sum_{j=0}^{m_{i}} u_{ij}t^{j},
\end{equation}
where $u_{ij}\in A$ satisfying
\[ |u_{ij}|_{\infty}<q^{\frac{s_{i}q}{q-1}}  \ \hbox{ and } \ u_{im_{i}} \neq 0   .\]
We put
\[ J_{\fs}:=\left\{0,1,\ldots,m_{1} \right\} \times \cdots \times \left\{0,1,\ldots,m_{r} \right\}     .\]
For each $\bj=(j_{1},\ldots,j_{r})\in J_{\fs}$, we set
\[ \bu_{\bj}:=(u_{1j_{1}},\ldots,u_{rj_{r}})\in A^{r},   \]
and
\[ a_{\bj}:=a_{\bj}(t):=t^{j_{1}+\cdots+j_{r}}    .\]
Set $\Gamma_{\fs}:=\Gamma_{s_{1}}\cdots\Gamma_{s_{r}}\in A$. The following formula is established in \cite{C14}.
\begin{theorem} \label{T:Chang}
For each $\fs=(s_{1},\ldots,s_{r})\in \NN^{r}$, we have that
\[ \zeta_{A}(\fs)=\frac{1}{\Gamma_{\fs}}\sum_{\bj \in J_{\fs}}a_{\bj}(\theta)\Li_{\fs}(\bu_{\bj}).  \]
\end{theorem}

\subsection{The main result}
Our main result is to show that the formula above is valid for the finite level:
\begin{theorem}\label{T:MainThm}

For each $\fs=(s_{1},\ldots,s_{r})\in \NN^{r}$, we have that
\[ \zeta_{\cA_{k}}(\fs)=\frac{1}{\Gamma_{\fs}}\sum_{\bj \in J_{\fs}}a_{\bj}(\theta)\Li_{\cA_{k},\fs}(\bu_{\bj}).  \]

\end{theorem}

For each nonnegative integer $i$, we let $A_{i+}$ be the set of all monic polynomials of degree $i$ in $A$.
For each $i \in \ZZ$ and $H = \sum u_{j} t^{j} \in k[t]$, we define $H^{(i)} := \sum u_{j}^{q^{i}} t^{j}$.
To prove the theorem above, we need the following interpolation formula of Anderson and Thakur~\cite{AT90}.
\begin{lemma}\label{L:AT}
Fixing $s\in \NN$, for any nonnegative integer $i$ we have
\[ \frac{H_{s-1}^{(i)}|_{t=\theta} }{L_{i}^{s}}=\Gamma_{s}\sum_{a\in A_{i+}}\frac{1}{a^{s}}     .\]
\end{lemma}

{\it{Proof of Theorem~\ref{T:MainThm}}}. It suffices to verify the identity for the $P$-component of the both sides of Theorem~\ref{T:MainThm} for primes $P$ with $\deg P \gg 0$. Let $P\in A_{+}$ satisfy $P\nmid \Gamma_{\fs}$. By definition, we have
\begin{eqnarray*}
\zeta_{\cA_{k}}(\fs)_{P} &=& \sum_{\substack{a_{1}, \dots, a_{r} \in A_{+} \\ \deg P > \deg a_{1} > \cdots > \deg a_{r}}} \dfrac{1}{a_{1}^{s_{1}} \cdots a_{r}^{s_{r}}} \hbox{ mod }P
= \sum_{\substack{\deg P > i_{1} > \cdots > i_{r} \geq 0 \\ a_{j} \in A_{i_{j} +}}} \dfrac{1}{a_{1}^{s_{1}} \cdots a_{r}^{s_{r}}} \hbox{ mod }P \\
&=& \sum_{\deg P > i_{1} > \cdots > i_{r} \geq 0} \ \sum_{a_{1} \in A_{i_{1} +}} \dfrac{1}{a_{1}^{s_{1}}}  \cdots \sum_{a_{r} \in A_{i_{r} +}} \dfrac{1}{a_{r}^{s_{r}}} \hbox{ mod }P \\
&=& \dfrac{1}{\Gamma_{\fs}} \sum_{\deg P > i_{1} > \cdots > i_{r} \geq 0} \dfrac{H_{s_{1}-1}^{(i_{1})}|_{t=\theta} \cdots H_{s_{r}-1}^{(i_{r})}|_{t=\theta}}{L_{i_{1}}^{s_{1}} \cdots L_{i_{r}}^{s_{r}}} \hbox{ mod }P,
\end{eqnarray*}
where the last equality comes from Lemma~\ref{L:AT}.

By (\ref{E:expansionHs}) we have
\begin{eqnarray*}
H_{s_{1}-1}^{(i_{1})}|_{t = \theta} \cdots H_{s_{r}-1}^{(i_{r})}|_{t = \theta}
&=& \sum_{j_{1} = 0}^{m_{1}} u_{1j_{1}}^{q^{i_{1}}} \theta^{j_{1}} \cdots \sum_{j_{r} = 0}^{m_{r}} u_{rj_{r}}^{q^{i_{r}}} \theta^{j_{r}} \\
&=& \sum_{\bj = (j_{1},\ldots,j_{r}) \in J_{\fs}} a_{\bj}(\theta) u_{1j_{1}}^{q^{i_{1}}} \cdots u_{rj_{r}}^{q^{i_{r}}}.
\end{eqnarray*}
It follows that
\begin{eqnarray*}
\zeta_{\cA_{k}}(\fs)_{P}
&=& \dfrac{1}{\Gamma_{\fs}} \sum_{\deg P > i_{1} > \cdots > i_{r} \geq 0} \ \sum_{\bj = (j_{1},\ldots,j_{r}) \in J_{\fs}} \dfrac{a_{\bj}(\theta) u_{1j_{1}}^{q^{i_{1}}} \cdots u_{rj_{r}}^{q^{i_{r}}}}{L_{i_{1}}^{s_{1}} \cdots L_{i_{r}}^{s_{r}}} \hbox{ mod }P \\
&=& \dfrac{1}{\Gamma_{\fs}} \sum_{\bj = (j_{1},\ldots,j_{r}) \in J_{\fs}} a_{\bj}(\theta) \sum_{\deg P > i_{1} > \cdots > i_{r} \geq 0} \dfrac{u_{1j_{1}}^{q^{i_{1}}} \cdots u_{rj_{r}}^{q^{i_{r}}}}{L_{i_{1}}^{s_{1}} \cdots L_{i_{r}}^{s_{r}}} \hbox{ mod }P \\
&=& \dfrac{1}{\Gamma_{\fs}} \sum_{\bj \in J_{\fs}} a_{\bj}(\theta) \Li_{\cA_{k},\fs}(\bu_{\bj})_{P},
\end{eqnarray*}
whence verifying Theorem~\ref{T:MainThm}.

\bibliographystyle{alpha}

\end{document}